\documentclass[11pt]{amsart}

\usepackage{mlmodern}

\usepackage{amsmath,amssymb}

\newtheorem{theorem}{Theorem}[section]
\newtheorem{corollary}[theorem]{Corollary}
\newtheorem{lemma}[theorem]{Lemma}
\newtheorem{proposition}[theorem]{Proposition}
\theoremstyle{remark}
\newtheorem{remark}[theorem]{Remark}

\newcommand{\R}{\mathbb R}
\newcommand{\norm}[1]{\lVert #1\rVert}
\newcommand{\E}{\mathbf{E}}
\newcommand{\Prob}{\mathbf{P}}

\title[]{A sharp density bound for infinite convolutions}

\author[]{Peio Alcorta}
\address{Department of Quantitative Methods, University of the Basque Country (EHU), Avda Lehendakari Aguirre 83, Bilbao, 48015, Spain}
\email{peio.alcorta@ehu.eus}

\subjclass[2020]{Primary 60E15; Secondary 60E10, 60G50}
\keywords{Infinite convolution, bounded density, convolution powers}

\begin{document}

\begin{abstract}
We prove a sharp $L^\infty$ bound for the density of an infinite convolution
in terms of $L^2$ norms of convolution powers of the summand densities. The
bound may be finite even when every summand density is unbounded and the
products arising from Young's inequality diverge.
\end{abstract}

\maketitle

\section{Introduction}

Sharp forms of Young's convolution inequality \cite{Beckner,BrascampLieb}
and maximum-density bounds for sums of independent random variables
\cite{BobkovChistyakovDensity,Rogozin} are classical, while the existence and regularity of infinite convolutions have long been studied
\cite{JessenWintner,Wintner}. To the best of our knowledge, no sharp $L^\infty$ bound for the density of an infinite convolution has
been given in terms of $L^2$ norms of convolution powers of the summand densities. We provide such a bound without requiring the summand densities to be bounded.

For illustration, let $X_k$ be independent random vectors in $\R^d$ with densities $f_k$, and suppose that their partial sums converge in distribution
to some $S$ with density $f_S$. For $f\in L^1(\R^d)$ and $n\geq1$, let $f^{*n}$ denote the $n$-fold convolution of $f$ with itself. Under suitable assumptions, a simple dyadic instance of the bound is
\[
 \norm{f_S}_\infty
 \leq
 \norm{f_1}_2\,
 \norm{f_2^{*2}}_2^{1/2}\,
 \norm{f_3^{*4}}_2^{1/4}\cdots.
\]

In fact, the bound may remain finite even when every summand density is unbounded and the products arising from Young's inequality and its sharp form diverge. Moreover, an example shows that the partial-sum densities may remain unbounded up to any prescribed finite index, while the limit has a bounded density. The estimate also has an information-theoretic interpretation in terms of R\'enyi entropy and yields an anti-concentration inequality.

\section{Density bounds for infinite convolutions}

\begin{lemma}\label{lem:endpoint}
Let $(\Omega,\mathcal F,\mu)$ be a $\sigma$-finite measure space. Let
$q_m\downarrow1$, and let $g_m,g$ be complex-valued measurable functions on $\Omega$ such 
that $g_m\to g$ almost everywhere. Then
\[
 \norm{g}_1\leq\liminf_{m\to\infty}\norm{g_m}_{q_m}.
\]
\end{lemma}

\begin{proof}
The assertion is immediate if $\mu(\Omega)=0$. Otherwise, choose measurable sets $A_j\uparrow\Omega$ with $0<\mu(A_j)<\infty$. For every
$j$, Fatou's lemma and H\"older's inequality on $A_j$ give
\[
 \int_{A_j}|g|\,d\mu
 \leq\liminf_{m\to\infty}\int_{A_j}|g_m|\,d\mu
 \leq\liminf_{m\to\infty}
\mu(A_j)^{1-1/q_m}\norm{g_m}_{q_m}.
\]
Since $\lim_{m\to\infty}\mu(A_j)^{1-1/q_m}=1$, we obtain $\int_{A_j}|g|\,d\mu \leq \liminf_{m\to\infty}\norm{g_m}_{q_m}$. The assertion follows by letting $j\to\infty$ and applying the monotone convergence theorem to $|g|\mathbf{1}_{A_j}\uparrow |g|$.
\end{proof}

\begin{theorem}\label{thm:main}
Let $X_k$ be independent random vectors in $\R^d$ with characteristic functions $\varphi_k(t)=\E e^{it\cdot X_k}$. Suppose that $S_N=\sum_{k=1}^N X_k$ converges in distribution to some $S$. Let $p_k>1$ satisfy $\sum_{k=1}^\infty p_k^{-1}=1$. If $\varphi_k\in L^{p_k}(\R^d)$, then
\begin{enumerate}
\item[(\textit{i})] The sequence $(P_m)$ with $P_m=\prod_{k=1}^m\norm{\varphi_k}_{p_k}$ either converges to a finite positive value or tends to infinity.
\item[(\textit{ii})] If $P_m$ converges, $S$ has a uniformly continuous density $f_S$ vanishing at infinity and satisfying the sharp bound
\[
 \norm{f_S}_\infty
 \leq (2\pi)^{-d}
 \prod_{k=1}^\infty\norm{\varphi_k}_{p_k}.
\]
\end{enumerate}
\end{theorem}

\begin{proof}
Let $\varphi_S$ be the characteristic function of $S$. Since $\varphi_S$ is
continuous with $\varphi_S(0)=1$, choose a ball $B$ centered at the origin with Lebesgue measure $|B|<1$ and $\inf_B|\varphi_S|>0$. By independence and convergence in distribution, $\prod_{k=1}^m\varphi_k\to\varphi_S$ pointwise. Since $|\varphi_k|\leq1$, on $B$ we have $|\varphi_k|\geq|\varphi_S|>0$ for every $k$, and $\sum_{k=1}^\infty\log|\varphi_k|=\log|\varphi_S|$. 

Restricting the $L^{p_k}$ norm to $B$ and applying Jensen's inequality to
$\log$ with respect to normalized Lebesgue measure on $B$, we obtain
\[
\log\norm{\varphi_k}_{p_k} \geq 
p_k^{-1}\log|B|+|B|^{-1}\int_B\log|\varphi_k(t)|\,dt.
\]
Since $|\varphi_k|\leq1$ and $|B|<1$, the right-hand side is negative. Thus the same bound holds with $\log\norm{\varphi_k}_{p_k}$ replaced by $a_k=\min\{0,\log\norm{\varphi_k}_{p_k}\}$. Summing over $k\leq m$ and using $\sum_{k=1}^m p_k^{-1}\leq 1$, we obtain
\[
\sum_{k=1}^m a_k \geq \log|B|+|B|^{-1}\int_B \sum_{k=1}^m \log|\varphi_k(t)|\,dt.
\]
Since $\log|\varphi_k|\leq 0$, on $B$ we have $\sum_{k=1}^m \log|\varphi_k(t)| \downarrow\log|\varphi_S(t)|$. Thus $\sum_{k=1}^m a_k$ is bounded below by a finite constant and is nonincreasing, since $a_k\leq 0$. Hence it converges. Moreover, 
\[
\log P_m-\sum_{k=1}^m a_k = \sum_{k=1}^m\max\{0,\log\norm{\varphi_k}_{p_k}\}
\]
is nondecreasing, and hence converges or tends to $+\infty$. Since $\sum_{k=1}^m a_k$ converges, $\log P_m$ converges or tends to $+\infty$, proving (\textit{i}).

To prove (\textit{ii}), suppose that $P_m$ converges and set $q_m^{-1}=\sum_{k=1}^m p_k^{-1}$. H\"older's inequality applied to the factors $|\varphi_k|^{q_m}$ with exponents $p_k/q_m$ gives $\norm{\prod_{k=1}^m\varphi_k}_{q_m} \leq P_m$. Since $q_m\downarrow1$ and $\prod_{k=1}^m\varphi_k\to\varphi_S$ pointwise,
Lemma~\ref{lem:endpoint} gives $\norm{\varphi_S}_1\leq
\prod_{k=1}^\infty\norm{\varphi_k}_{p_k}$. Thus $\varphi_S\in L^1(\R^d)$. By Fourier inversion and the Riemann--Lebesgue lemma, $S$ has a uniformly continuous density vanishing at infinity, with $\norm{f_S}_\infty\leq(2\pi)^{-d}\norm{\varphi_S}_1$.

For sharpness, we show that equality is attained for suitable Gaussian summands. Let $X_k\sim\mathcal N(0,p_k^{-1}V)$ be independent, where $V$ is a symmetric positive definite matrix. Since $\varphi_k(t)=e^{-t\cdot Vt/(2p_k)}$, we have $\prod_{k=1}^\infty\varphi_k(t)=e^{-t\cdot Vt/2}$. Hence $S_N$ converges in distribution to $S\sim\mathcal N(0,V)$, and $\varphi_S=\varphi_k^{p_k}$ for every $k$. Therefore $\norm{\varphi_S}_1=\prod_{k=1}^\infty\norm{\varphi_k}_{p_k}$. On the other hand, $\norm{f_S}_\infty=f_S(0)$ since $f_S$ is a centered Gaussian, while
$f_S(0)=(2\pi)^{-d}\norm{\varphi_S}_1$ by Fourier inversion and the positivity of $\varphi_S$. Hence $\norm{f_S}_\infty=(2\pi)^{-d}\prod_{k=1}^\infty\norm{\varphi_k}_{p_k}$.
\end{proof}

Plancherel's theorem rewrites the preceding estimate in terms of convolution powers. The resulting bound may be finite even when every $f_k$ is unbounded; see Section~\ref{sec:delayed}.

\begin{theorem}\label{thm:convolution}
Let $X_k$ be independent random vectors in $\R^d$ with densities $f_k$. Suppose that $S_N=\sum_{k=1}^N X_k$ converges in distribution to some $S$. Let $n_k$ be positive integers satisfying $\sum_{k=1}^\infty (2n_k)^{-1}=1$. If $f_k^{*n_k}\in L^2(\R^d)$, then 
    \begin{enumerate}
\item[(\textit{i})] The sequence $(Q_m)$ with $Q_m=\prod_{k=1}^m\norm{f_k^{*n_k}}_2^{1/n_k}$ either converges to a finite positive value or tends to infinity.
\item[(\textit{ii})] If $Q_m$ converges, $S$ has a uniformly continuous density $f_S$ vanishing at infinity and satisfying the sharp bound
\[
 \norm{f_S}_\infty
 \leq
 \prod_{k=1}^\infty\norm{f_k^{*n_k}}_2^{1/n_k}.
\]
\end{enumerate}
\end{theorem}

\begin{proof}
Let $\varphi_k$ be the characteristic function of $X_k$. Since the density $f_k^{*n_k}$ belongs to $L^1(\R^d)\cap L^2(\R^d)$ and has characteristic function $\varphi_k^{n_k}$, Plancherel's theorem gives
\[
\norm{\varphi_k}_{2n_k}
=(2\pi)^{d/(2n_k)}\norm{f_k^{*n_k}}_2^{1/n_k}.
\]
Let $P_m=\prod_{k=1}^m\norm{\varphi_k}_{2n_k}$ and $c_m=\sum_{k=1}^m(2n_k)^{-1}$. Then $P_m=(2\pi)^{c_m d}Q_m$. Since $c_m\uparrow1$, (\textit{i}) follows from Theorem~\ref{thm:main}(\textit{i}) applied with $p_k=2n_k$. If $Q_m\to Q$, then $P_m\to(2\pi)^d Q$, and Theorem~\ref{thm:main}(\textit{ii}) yields $\norm{f_S}_\infty\leq Q$.

For sharpness, specialize the Gaussian case of Theorem~\ref{thm:main} with $p_k=2n_k$ by taking $X_k\sim \mathcal N\!\left(0, (2n_k)^{-1}V\right)$. Then $f_k^{*n_k}$ is the $\mathcal N(0,V/2)$ density for every $k$. Denote this common density by $\phi$. Since $\sum_{k=1}^\infty n_k^{-1}=2$ and $f_S=\phi^{*2}$, we have
\[
\prod_{k=1}^\infty\norm{f_k^{*n_k}}_2^{1/n_k}
=\norm{\phi}_2^2=f_S(0)=\norm{f_S}_\infty.\qedhere
\]
\end{proof}

A particularly simple admissible choice of the sequence $(n_k)$ is the dyadic
one, for which the bound takes an explicit form.

\begin{corollary}[Dyadic form]\label{cor:dyadic} 
In the setting of Theorem~\ref{thm:convolution}, let $n_k=2^{k-1}$. If the product below converges, then the following sharp bound holds:
\[
 \norm{f_S}_\infty
 \leq \norm{f_1}_2\,
       \norm{f_2^{*2}}_2^{1/2}\,
       \norm{f_3^{*4}}_2^{1/4}\cdots.
\]
\end{corollary}

Our proof of Theorem~\ref{thm:main} uses H\"older's inequality only for finitely many factors together with Lemma~\ref{lem:endpoint}, and in addition
yields the dichotomy in Theorem~\ref{thm:main}(\textit{i}) for the partial products $P_m$. The estimate in Theorem~\ref{thm:main}(\textit{ii}) alone can alternatively be obtained from countable versions of H\"older's inequality \cite{ChenJiaJiao,Karakostas}, applied to
$|\varphi_S|=\prod_{k=1}^\infty|\varphi_k|$. For finitely many summands with
a common exponent, the corresponding Fourier estimate appears in the proof of
\cite[Prop.~3]{BobkovGotze}.

Theorem~\ref{thm:convolution} immediately yields a sharp anti-concentration inequality, connecting the estimate with the classical theory of concentration functions for sums of independent random variables; see \cite{Petrov}.

\begin{corollary}[Anti-concentration]
In the setting of Theorem~\ref{thm:convolution}, if the product below converges, then
\[
\Prob(S\in A)
\leq
|A|\prod_{k=1}^\infty\norm{f_k^{*n_k}}_2^{1/n_k}
\]
for every Borel set $A\subset\R^d$ with $|A|<\infty$. The bound is sharp.
\end{corollary}

\begin{proof}
The bound is immediate. For sharpness, consider the Gaussian equality case from the proof of Theorem~\ref{thm:convolution}. Let $B_\varepsilon$ be the ball of radius $\varepsilon$ centered at the origin. Then, as $\varepsilon\downarrow0$,
\[
\Prob(S\in B_\varepsilon)/|B_\varepsilon|\to f_S(0)=\norm{f_S}_\infty.\qedhere
\]
\end{proof}

The hypotheses of Theorem~\ref{thm:convolution} involve the convolution powers $f_k^{*n_k}$. Young's inequality yields the following sufficient condition in terms of the $f_k$, at the cost of a coarser bound.

\begin{corollary}[Young-type criterion]\label{cor:young}
Assume the hypotheses of Theorem~\ref{thm:convolution}, except for the condition $f_k^{*n_k}\in L^2(\R^d)$. Set $r_k=2n_k/(2n_k-1)$ and suppose that $f_k\in L^{r_k}(\R^d)$ for every $k$. Let $R_m=\prod_{k=1}^m\norm{f_k}_{r_k}$. If the partial products $R_m$ converge, then $S$ has a uniformly continuous density vanishing at infinity and
\[
 \norm{f_S}_\infty\leq\prod_{k=1}^\infty \norm{f_k}_{r_k}.
\]
\end{corollary}

\begin{proof}
Set $Q_m=\prod_{k=1}^m\norm{f_k^{*n_k}}_2^{1/n_k}$. Young's inequality gives $f_k^{*n_k}\in L^2(\R^d)$ with $\norm{f_k^{*n_k}}_2^{1/n_k}\leq\norm{f_k}_{r_k}$, and hence $Q_m\leq R_m$ for every $m$. Since $R_m$ converges to some $R$, $Q_m$ is bounded, and therefore converges to some $Q$ by Theorem~\ref{thm:convolution}(\textit{i}). Passing to the limit gives $Q\leq R$, while $\norm{f_S}_\infty\leq Q$ follows from Theorem~\ref{thm:convolution}(\textit{ii}).
\end{proof}

The construction in Section~\ref{sec:delayed} shows that $R_m$ may tend to infinity while $Q_m$ converges, so Theorem~\ref{thm:convolution} may apply when Corollary~\ref{cor:young} does not.

\begin{remark}[Sharp Young]
Using the sharp form of Young's convolution inequality, obtained independently by Beckner \cite{Beckner} and by Brascamp and Lieb \cite{BrascampLieb}, the sufficient condition above can be improved
by replacing each factor $\norm{f_k}_{r_k}$ by
$C_{r_k}^{\,d}\norm{f_k}_{r_k}$, where
$C_r^2=r^{1/r}/(r')^{1/r'}$ and $1/r+1/r'=1$.
As shown in Section~\ref{sec:delayed}, even this sharper sufficient condition may fail: the factors $C_{r_k}^{\,d}\norm{f_k}_{r_k}$ need not form a convergent infinite product, even though the convolution-power product in Theorem~\ref{thm:convolution} converges.
\end{remark}

\section{Arbitrarily delayed boundedness}\label{sec:delayed}

We next show that boundedness of the partial-sum densities can be delayed arbitrarily long, even though the limiting density is bounded. For related boundedness phenomena for convolution powers of unbounded densities, see \cite{PestmanTuerlinckxVanpaemel}.

\begin{proposition}[Delayed boundedness]\label{prop:gamma}
For every $N\geq1$, there exist independent random variables $X_k$ with
unbounded densities $f_k$ such that, for each $p>1$, there exists a $K_p$
for which $f_k\notin L^p(\R)$ whenever $k\geq K_p$. Moreover,
$S_m=\sum_{k=1}^m X_k$ has an unbounded density exactly when $m\leq N$, and it converges almost surely to some $S$ with a bounded continuous density.
\end{proposition}

\begin{proof}
Fix $N\geq1$ and write $\delta=2^{-N}$. Let $X_k$ be independent gamma random variables
with shape $\alpha_k=(1+\delta)2^{-k}$ and rate $1$. For $x>0$, their densities are
\[
 f_k(x)=\frac{x^{\alpha_k-1}e^{-x}}{\Gamma(\alpha_k)},
 \qquad\text{with}\qquad
 \norm{f_k}_p^p
 =
 \frac{\Gamma(p\alpha_k-p+1)}
 {p^{\,p\alpha_k-p+1}\Gamma(\alpha_k)^p}
\]
for $p>1$ whenever $p\alpha_k-p+1>0$, equivalently $\alpha_k>1-p^{-1}$, which is exactly the condition for $f_k\in L^p(\R)$. Since $\alpha_k\downarrow0$, for each $p>1$ there exists
a $K_p$ such that $f_k\notin L^p(\R)$ whenever $k\geq K_p$. Moreover,
$\alpha_k\leq3/4<1$ for all $k\geq1$, so $f_k(x)\to\infty$ as $x\downarrow0$.

Next, independent gamma variables with a common rate add their shapes, so $S_m$
has the gamma distribution with shape
$\theta_m=(1+\delta)(1-2^{-m})$. Since $X_k\geq0$ a.s., $S_m$ is nondecreasing a.s. Set $S=\sup_m S_m$. Then $S_m\uparrow S$ a.s., and by monotone convergence,
\[
\E S=\lim_{m\to\infty}\E S_m
=\lim_{m\to\infty}\theta_m=1+\delta<\infty,
\]
so $S<\infty$ a.s. Since $\theta_m\uparrow1+\delta$, the distributions of $S_m$ converge to the
gamma distribution with shape $1+\delta$, and hence $S$ has this distribution, whose density is bounded and continuous because $1+\delta>1$.

Finally, if $m\leq N$, then $\theta_m\leq1-\delta^2<1$, so the density of
$S_m$ is unbounded. If $m\geq N+1$, then
$\theta_m\geq(1+\delta)(1-\delta/2)>1$, so the density of $S_m$ is bounded and continuous.
\end{proof}

\begin{remark}[Beyond Young]
For the gamma example above, let $n_k=2^{k-1}$. Then every $f_k^{*n_k}$ is the same gamma density
$\gamma$, with shape $(1+\delta)/2>1/2$. Hence $\gamma\in L^2(\R)$ and $\prod_{k=1}^\infty \norm{f_k^{*n_k}}_2^{1/n_k}=\norm{\gamma}_2^2<\infty$, so Corollary~\ref{cor:dyadic} applies. On the other hand, the example lies outside the sufficient condition of
Corollary~\ref{cor:young}. Indeed, for $r_k=2^k/(2^k-1)$, set
$b_k=r_k\alpha_k-r_k+1$. Then
$b_k=\delta/(2^k-1)>0$, so $f_k\in L^{r_k}(\R)$. Since $0<\alpha_k<1$ and $\Gamma$ is decreasing on $(0,1]$,
we have $1=\Gamma(1)\leq\Gamma(\alpha_k)$. Moreover, since $\Gamma$ is bounded on $[1,2]$, there is a constant $M$ such that $\alpha_k\Gamma(\alpha_k)=\Gamma(1+\alpha_k)\leq M$. Hence $1\leq\Gamma(\alpha_k)\leq M2^k/(1+\delta)$. Raising both sides to the power $r_k-1=(2^k-1)^{-1}$, the resulting upper bound tends to $1$, and therefore $\Gamma(\alpha_k)^{r_k-1}\to1$. 
Also, $1\leq r_k^{b_k}\leq r_k$ with $r_k\downarrow1$, while $x\Gamma(x)\to1$ as $x\downarrow0$. The formula above then yields $\norm{f_k}_{r_k}^{r_k}\to(1+\delta)/\delta$. Since $C_{r_k}\to1$, both $\norm{f_k}_{r_k}$ and
$C_{r_k}\norm{f_k}_{r_k}$ converge to $(1+\delta)/\delta>1$. Thus the products arising from the classical and sharp forms of Young's inequality both diverge.
\end{remark}

\section{An information-theoretic interpretation}

The bound in Theorem~\ref{thm:convolution} admits an information-theoretic
interpretation. For related inequalities for convolutions and finite sums, see \cite{BobkovChistyakov,ChasapisSinghTkocz,DemboCoverThomas,XuMelbourneMadiman}. For a random vector $Y$ with density $\rho$,
let $h_2(Y)=-2\log\norm{\rho}_2$ and $h_\infty(Y)=-\log\norm{\rho}_\infty$
denote its R\'enyi entropies of orders $2$ and $\infty$, the latter also known as min-entropy.

\begin{corollary}[R\'enyi entropy]\label{cor:renyi}
Let $X_k$ be independent random vectors in $\R^d$ with densities $f_k$, and
suppose that $S_N=\sum_{k=1}^N X_k$ converges in distribution to some $S$. Let $n_k$ be positive
integers satisfying $\sum_{k=1}^\infty(2n_k)^{-1}=1$ and
$f_k^{*n_k}\in L^2(\R^d)$. Let $Y_k$ be the sum of $n_k$ independent copies of $X_k$, and set
$\lambda_k=(2n_k)^{-1}$. Then
\begin{enumerate}
\item[(\textit{i})] The sequence $(H_m)$ with 
$H_m=\sum_{k=1}^m\lambda_k h_2(Y_k)$ either converges to a finite value or tends to $-\infty$.
\item[(\textit{ii})] If $H_m$ converges, $S$ satisfies the sharp bound
\[
h_\infty(S)\geq\sum_{k=1}^\infty\lambda_k h_2(Y_k).
\]
\end{enumerate}
\end{corollary}

\begin{proof}
Let $Q_m=\prod_{k=1}^m\norm{f_k^{*n_k}}_2^{1/n_k}$. Since
$H_m=-\log Q_m$, (\textit{i}) follows from
Theorem~\ref{thm:convolution}(\textit{i}). If $H_m\to H$, then $Q_m\to e^{-H}$, and
Theorem~\ref{thm:convolution}(\textit{ii}) gives
$\norm{f_S}_\infty\leq e^{-H}$. Hence
$h_\infty(S)\geq H$, proving (\textit{ii}). In the Gaussian case of Theorem~\ref{thm:convolution}, each $Y_k$ has the common density $\phi$ and $\norm{f_S}_\infty=\norm{\phi}_2^2$, so
$h_2(Y_k)=h_\infty(S)$ for every $k$. Since $\sum_{k=1}^\infty\lambda_k=1$, equality holds.
\end{proof}

\section*{Funding}

This research did not receive any specific grant from funding agencies in the public, commercial, or not-for-profit sectors.

\section*{Declaration of generative AI and AI-assisted technologies in the manuscript preparation process}

During the preparation of this work, the author used OpenAI's ChatGPT and Anthropic's Claude as auxiliary tools for mathematical development and the editorial preparation of the manuscript. The author directed the use of these tools; suggestions produced by them were critically assessed, revised, or discarded by the author, who takes full responsibility for the content of the published article.


\begin{thebibliography}{99}

\bibitem{Beckner}
W.~Beckner,
\emph{Inequalities in Fourier analysis},
Ann. of Math. (2) \textbf{102} (1975), no.~1, 159--182.


\bibitem{BobkovChistyakovDensity}
S.~G. Bobkov and G.~P. Chistyakov,
\emph{Bounds on the maximum of the density for sums of independent random variables},
J. Math. Sci. (N.Y.) \textbf{199} (2014), no.~2, 100--106.


\bibitem{BobkovChistyakov}
\bysame,
\emph{Entropy power inequality for the R\'enyi entropy},
IEEE Trans. Inform. Theory \textbf{61} (2015), no.~2, 708--714.

\bibitem{BobkovGotze}
S.~G. Bobkov and F.~G\"otze,
\emph{Berry--Esseen bounds in the local limit theorems},
Lithuanian Math. J. \textbf{65} (2025), no.~1, 50--66.

\bibitem{BrascampLieb}
H.~J. Brascamp and E.~H. Lieb,
\emph{Best constants in Young's inequality, its converse, and its
generalization to more than three functions},
Adv. Math. \textbf{20} (1976), no.~2, 151--173.


\bibitem{ChasapisSinghTkocz}
G.~Chasapis, S.~Singh, and T.~Tkocz,
\emph{Entropies of sums of independent gamma random variables},
J. Theoret. Probab. \textbf{36} (2023), no.~2, 1227--1242.

\bibitem{ChenJiaJiao}
W.~Chen, L.~B. Jia, and Y.~Jiao,
\emph{H\"older's inequalities involving the infinite product and their
applications in martingale spaces},
Anal. Math. \textbf{42} (2016), no.~2, 121--141.

\bibitem{DemboCoverThomas}
A.~Dembo, T.~M. Cover, and J.~A. Thomas,
\emph{Information theoretic inequalities},
IEEE Trans. Inform. Theory \textbf{37} (1991), no.~6, 1501--1518.


\bibitem{JessenWintner}
B.~Jessen and A.~Wintner,
\emph{Distribution functions and the Riemann zeta function},
Trans. Amer. Math. Soc. \textbf{38} (1935), no.~1, 48--88.

\bibitem{Karakostas}
G.~L. Karakostas,
\emph{An extension of H\"older's inequality and some results on infinite
products},
Indian J. Math. \textbf{50} (2008), no.~2, 303--307.

\bibitem{PestmanTuerlinckxVanpaemel}
W.~Pestman, F.~Tuerlinckx, and W.~Vanpaemel,
\emph{Persistently unbounded probability densities},
Statist. Probab. Lett. \textbf{118} (2016), 135--138.

\bibitem{Petrov}
V.~V. Petrov,
\emph{Sums of independent random variables},
Ergebnisse der Mathematik und ihrer Grenzgebiete, Band 82,
Springer-Verlag, Berlin--Heidelberg, 1975.


\bibitem{Rogozin}
B.~A. Rogozin,
\emph{Estimation of the maximum of a convolution of bounded densities},
Theory Probab. Appl. \textbf{32} (1988), no.~1, 48--56.



\bibitem{Wintner}
A.~Wintner,
\emph{On the densities of infinite convolutions},
Amer. J. Math. \textbf{59} (1937), no.~2, 376--378.

\bibitem{XuMelbourneMadiman}
P.~Xu, J.~Melbourne, and M.~Madiman,
\emph{Infinity-R\'enyi entropy power inequalities},
2017 IEEE International Symposium on Information Theory (ISIT),
IEEE, 2017, pp.~2985--2989.


\end{thebibliography}
\end{document}